\documentclass{amsart}
\usepackage[utf8]{inputenc}
\usepackage[T1]{fontenc}
\usepackage{amssymb, amsmath, amsthm, amsfonts,url}
\usepackage[alphabetic,backrefs,lite]{amsrefs}
\usepackage{amscd}  
\usepackage{stmaryrd}
\usepackage{mathrsfs,mathtools}
\usepackage[all]{xy} 
\usepackage{float}
\usepackage{MnSymbol}
\usepackage{multirow} 
\usepackage{array}
\usepackage{faktor}
\usepackage[shortlabels]{enumitem}
\usepackage{multicol}
\usepackage[usenames,dvipsnames]{xcolor}
\usepackage{tikz-cd}

\newcommand{\C}{\mathbb C}

\newcommand{\Q}{\mathbb Q}
\newcommand{\G}{\mathbb G}

\newcommand{\R}{\mathbb R}

\newcommand{\Vl}{V_{\ell}}
\newcommand{\Ql}{\Q_{\ell}}

\DeclareMathOperator{\alg}{alg}
\DeclareMathOperator{\GlK}{G_{\ell,K}^{\alg}}
\DeclareMathOperator{\GlKo}{G_{\ell,K,1}^{\alg}}
\DeclareMathOperator{\GlKe}{G_{\ell,K_e}^{\alg}}
\DeclareMathOperator{\GlKeo}{G_{\ell,K_e,1}^{\alg}}
\DeclareMathOperator{\gpast}{ast}
\DeclareMathOperator{\AST}{AST}

\DeclareMathOperator{\Aut}{Aut}

\DeclareMathOperator{\gpDL}{DL}

\DeclareMathOperator{\End}{End}

\DeclareMathOperator{\Gal}{Gal}

\DeclareMathOperator{\GL}{GL}
\DeclareMathOperator{\SL}{SL}

\DeclareMathOperator{\GSp}{GSp}

\DeclareMathOperator{\id}{id}

\DeclareMathOperator{\Ker}{Ker}

\DeclareMathOperator{\MMT}{MMT}

\DeclareMathOperator{\MT}{MT}
\DeclareMathOperator{\Hg}{Hg}

\DeclareMathOperator{\Sp}{Sp}

\DeclareMathOperator{\SU}{SU}
\DeclareMathOperator{\Nor}{N}
\DeclareMathOperator{\Uni}{U}
\DeclareMathOperator{\ST}{ST}

\newtheorem{theorem}{Theorem}[section]

\newtheorem{corollary}[theorem]{Corollary}

\newtheorem{conjecture}[theorem]{Conjecture}

\theoremstyle{definition}
\newtheorem{definition}[theorem]{Definition}
\newtheorem{example}[theorem]{Example}

\theoremstyle{remark}
\newtheorem{remark}[theorem]{\bf Remark}
\newtheorem{notation}[theorem]{\bf Notation}

\numberwithin{equation}{section}

\begin{document}

\title[A remark on the component group of the Sato--Tate group]{A remark on the component group of the Sato--Tate group}

\author[Grzegorz Banaszak]{Grzegorz Banaszak}
\address{Department of Mathematics and Computer Science, Adam Mickiewicz University,
Pozna\'{n} 61-614, Poland} 
\email{banaszak@amu.edu.pl}

 %   Information for second author
\author[Victoria Cantoral Farf\'an]{Victoria Cantoral Farf\'an}
\address{Mathematisches Institut, Georg-August Universit\"at G\"ottingen, Bunsenstrasse 3-5, 37073 G\"ottingen, Germany}
\email{victoria.cantoralfarfan@mathematik.uni-goettingen.de}

%--------------------Introduction-----------------------------------

%    General info
\date{\today}


\thanks{Cantoral Farf\'an was partially supported by KU Leuven IF C14/17/083. The authors acknowledge the support provided by the Department of Mathematics and Computer Science of Adam Mickiewicz University in Poznań during the research visit of Cantoral Farf\'an.} 

\begin{abstract}
In this paper we give a complete characterization of the component group of the Sato--Tate group of an abelian variety $A$ of arbitrary dimension, defined over a number field $K,$ in terms of the connectedness of the Lefschetz group associated to $A.$  
\end{abstract}

\maketitle

%--------------------Introduction-----------------------------------

\section{Introduction}\label{sec: section 1}

	In the past six decades, the Sato--Tate conjecture has been largely studied since it has a strong connection with the generalized Riemann hypothesis. The Sato--Tate conjecture for an elliptic curve without complex multiplication defined over a number field, predicts the equidistribution of traces of Frobenius automorphisms with respect to the Haar measure of the corresponding Sato--Tate group. On the other hand the generalized Riemann hypothesis predicts where are located the zeros of the $L$-function, associated with the elliptic curve.
	
Despite all the efforts, only a few cases are known so far~\cite{Taylor 2008},~\cite{allen2018potential}. It is known, that when considering elliptic curves over number fields, the Sato--Tate group can only be one of the following three groups $\SU(2),\, \Uni(1)$ and $\Nor(\Uni(1))$. Nonetheless, when considering higher-dimensional abelian varieties, the classification of the possible Sato--Tate groups become more difficult to study~\cite{FKRS12} and~\cite{FKS2019}.  It is of utmost importance to determine Sato--Tate groups. The first approach is to study their component groups rather than the groups themselves in some particular settings. The goal of this paper is to give an explicit determination of the component group of the Sato--Tate group.

\subsection{The main result}

	The algebraic Sato--Tate conjecture, introduced by Serre and developed later on by Banaszak and Kedlaya~\cites{BK1, BK2}, is a key input in the direction of determining the Sato--Tate group. Notice that once we have a better understanding of the Sato--Tate group we could potentially determine new instances of the generalized Sato--Tate conjecture for higher-dimensional abelian varieties. This latter conjecture predicts the equidistribution of the normalized factors of the $L$-function associated to an abelian variety. One of the main goals of this paper is to give a characterization of the component group of the Sato--Tate group when enough information is provided. 
	The initial motivation of this paper was a previous result of Banaszak and Kedlaya~\cite[Thm. 6.1]{BK1}. In this theorem, the authors gave very specific conditions for the algebraic Sato--Tate conjecture to hold. More precisely, the algebraic Sato--Tate conjecture is known to be true when: 1) the abelian variety is fully of Lefschetz type (i.e. the Mumford--Tate conjecture is satisfied and the Hodge group coincides with the Lefschetz group), and 2) the twisted Lefschetz group is connected. Notice that the second condition does not always hold once $g>3$. Hence, we wonder what can be said when the second condition is not anymore satisfied. The spirit of our main result lies in the fact that in the case of simple abelian varieties of type III, in the sense of Albert's classification, the twisted Lefschetz group is not connected. Let us recall that this type in Albert's classification only appears when the dimension of the abelian variety is at least $4$. Our main result can be stated as follows:

\begin{theorem}\label{th: theorem intro}
Let $A$ be an abelian variety defined over a number field $K$. Consider its twisted Lefschets group $\gpDL_K(A)$ and the smallest finite field extension $K_e /K$ where all endomorphisms of $A$ are defined over. Under some hypotheses\footnote{We refer the reader to Section~\ref{sec: section 3}, Theorem~\ref{th: 3.4}} we obtain the following direct product decomposition:
\begin{equation}
	\pi_{0} (\gpDL_{K}(A))= \pi_{0} (\gpDL_{K}^{\id} (A))  \times \Gal(K_{e}/K).
\end{equation}
\end{theorem}	

\noindent 
Notice that, when assuming that the abelian variety is fully of Lefschetz type\footnote{See Definition~\ref{def: FLT}.}, we have the following isomorphisms: 
\begin{equation}
	\pi_0(\GlKo_{\Ql}) \simeq \pi_0(\AST_K(A)) \simeq \pi_0(\ST_K(A)) \simeq \pi_0(\gpDL_{K}(A)). 
\end{equation}
Hence, we are able to concretely determine, as a corollary, the component group of the Sato--Tate group in terms of the connectedness of the Lefschetz group.

\begin{corollary}
Under the hypotheses of Theorem~\ref{th: theorem intro} and the assumption that $A$ is fully of Lefschetz type, we have the following isomorphism:
\begin{equation}
	\pi_0(\ST_K(A))  \simeq \pi_0(\gpDL_{K}(A)) \simeq \pi_0 \gpDL_K^{id}(A) \times Gal(K_e/K).
	\label{isomorphisms-introduction}
\end{equation}	
In particular if $\gpDL_K^{id}(A)$ is connected then: 
\begin{equation}
	\pi_0(\ST_K(A)) \simeq \pi_0(\gpDL_{K}(A)) \simeq \Gal(K_e/K).
\end{equation}
\end{corollary}

\noindent
In other words, we provide a way to determine the component group of the Sato--Tate group in terms of endomorphisms of $A$ when enough information is known. We refer the reader to Section~\ref{sec: section 2} for the notation introduced above. In Section~\ref{sec: section 3} we present a proof of our main result.

%--------------------Section 2-----------------------------------

\section{Preliminaries}\label{sec: section 2}

	Throughout this paper $A$ will denote an abelian variety defined over a number field $K$ with the rational endomorphism algebra $D := \End_{{\overline K}} (A)\otimes \Q$. Let $V$ denote the first homology group $V(A) := H_1 (A(\C), \, \Q)$ and let $\Vl$ denote the $\Ql$-vector space $\Vl(A) := V(A) \otimes_{\Q} \Ql$ for every prime number $\ell$. Let $G_K:=\Gal(\overline{K}/K)$ be the absolute Galois group and 
\begin{equation}\label{eq: l-adic representation}
	\rho_\ell \, :\, G_K \rightarrow \GL(\Vl)
\end{equation}
the $\ell$-adic representation attached to $A$. Consider as well the $\Q$-bilinear non degenerate alternating form $\psi: V\times V \to \Q$ coming from the polarization of $A$. Equivalently, for every prime number $\ell$, we have $\psi_{\ell}: \Vl\times \Vl \to \Ql$.

Recall the definition of general symplectic groups associated with forms $\psi$ and $\psi_{\ell}:$
\begin{equation}
	\begin{aligned}
		\GSp_{(V,\psi)}&:=\{g\in \GL(V);\; \psi(gv,gw)=\chi(g)\psi(v,w),\; \forall v,w \in V\},\\
		\GSp_{(\Vl,\psi_{\ell})}&:=\{g\in \GL(\Vl);\; \psi_{\ell}(gv,gw)=\chi_{\ell}(g)\psi(v,w),\; \forall v,w \in \Vl \},
	\end{aligned}
\end{equation}
where $\chi$ and $\chi_{\ell}$ are the following associated characters: 
\begin{equation}
	\begin{aligned}
		\chi:\GSp_{(V,\psi)}&\to \G_{m,\Q}\\
		\chi_{\ell}:\GSp_{(\Vl,\psi_{\ell})}&\to \G_{m,\Ql}.
	\end{aligned}
\end{equation}
The special symplectic groups are given by
\begin{equation}
	\begin{aligned}
		\Sp_{(V,\psi)}&:=\{g\in \GL(V);\; \psi(gv,gw)=\psi(v,w),\; \forall v,w \in V\},\\
		\Sp_{(\Vl,\psi_{\ell})}&:=\{g\in \GL(\Vl);\; \psi_{\ell}(gv,gw)=\psi(v,w),\; \forall v,w \in \Vl\}.
	\end{aligned}
\end{equation}

%--------------------Subsection 2.1------------------------------

\subsection{The endomorphism field extension}

	When studying abelian varieties, defined over a number field $K$, we will examine a particular finite field extension of $K$ which is related to the endomorphisms of the abelian variety. Consider the following continuous representation~\cites{BK1,BK2}
\begin{equation}\label{eq: representation rho_e}
	\rho_e: G_K  \to \Aut_{\Q}(D)
\end{equation}
and denote by $K_e$ the fixed field of the kernel of $\rho_e$, that is $K_e:=\overline{K}^{\Ker \rho_e}$.
\begin{notation}
The finite field extension $K_e$ will be called the \emph{endomorphism field extension}.
\end{notation}
	
%--------------------Subsection 2.2------------------------------	

\subsection{The Lefschetz group and the twisted Lefschetz group}

\begin{definition}
The \emph{Lefschetz group} of an abelian variety $A$ is defined as follows: 
\begin{equation}\label{The Lefschetz group}
	\mathcal{L}(A):=C_D(\Sp_{(V,\psi)})^{\circ}. 
\end{equation}
\end{definition}
For each $\tau \in \Gal(K_{e}/K)$ we have a closed subscheme of $\Sp_{(V,\, \psi)}:$ 
\begin{equation}
	\gpDL_{K}^{\tau}(A) := \{g \in \Sp_{(V,\, \psi)}: \, g \beta g^{-1} = \rho_e(\tau) (\beta) \,\,\, \forall \,\beta \in D\}.
\end{equation}
Taking disjoint union, we obtain the {\it{twisted Lefschetz group}}~\cite[Def. 5.2]{BK1}:
\begin{equation}
	\gpDL_{K}(A) \,\,\, : = \bigsqcup_{\tau \in \Gal(K_{e}/K)} \,\, \gpDL_{K}^{\tau}(A).
\end{equation}

\begin{remark}\label{rmk: relations TLG}
We have the following equalities (see~\cite[\S 5]{BK1},~\cite[\S 3]{BK2}):
\begin{enumerate}
\item $\gpDL^{id}_K(A)=\{g\in \Sp_{(V,\psi)}, \, g\beta g^{-1}=\beta, \, \forall \beta\in D\}=C_D(\Sp_{(V,\psi)}),$
\item $\gpDL^{id}_K(A)^{\circ}=C_D(\Sp_{(V,\psi)})^{\circ}=\mathcal{L}(A)$,
\item $\gpDL_K(A)^{\circ}=\gpDL^{id}_K(A)^{\circ}$,
\item $\gpDL^{id}_K(A)=\gpDL^{id}_{K_e}(A)=\gpDL_{K_e}(A)$.
\end{enumerate}
\end{remark}

%--------------------Subsection 2.3------------------------------	

\subsection{The $\ell$-adic monodromy groups and their twists}

	For the representation~\eqref{eq: l-adic representation} define the $\ell$-adic monodromy group $\GlK(\Vl,\psi_{\ell})$ as the Zariski closure of the image of the $\ell$-adic representation in $\GSp_{(\Vl, \psi_{\ell})}$.
When no ambiguities exist regarding the $\Ql$-vector space, we will denote the $\ell$-adic monodromy group by $\GlK$ rather than $\GlK(\Vl,\psi_{\ell})$. 
Let us define the following algebraic group defined over $\Ql$:
\begin{equation}
	\GlKo=\GlKo(\Vl,\psi_{\ell}):=\GlK \cap \Sp_{(\Vl,\psi_{\ell})}.
\end{equation}

For all $\tau\in \Gal(K_e/K)$ consider the following twists of the monodromy group.
\begin{equation}
	\begin{aligned}
		(\GlK)^{\tau} &:= \{g \in \GlK: \, g \beta g^{-1} = \rho_e(\tau) (\beta) \,\,\,\, \forall \,\beta \in D\},\\
		(\GlKo)^{\tau} &:= (\GlK)^{\tau} \cap \GlKo.
	\end{aligned}
	\label{eq:eq 2.8}
\end{equation}
Notice that we have
\begin{equation}
	(\GlKo)^{\tau}  \subseteq \gpDL_{K}^{\tau}(A)_{\Ql}.
	\label{GlK1algtau subset DLKtauVpsiDQl}
\end{equation}
	
Let $\tilde\tau \in G_K$ be a lift of $\tau \in \Gal(K_{e}/K)$ and remark that the Zariski closure of $\rho_\ell (\tilde\tau \, G_{K_{e}}) = \rho_\ell (\tilde\tau) \, \rho_\ell (G_{K_{e}})$ in $\GSp_{(\Vl, \psi_\ell)}$ is  $\rho_\ell (\tilde\tau) \, \GlKe$, where $G_{K_e}$ is the absolute Galois group of the endomorphism field $K_e$.
Moreover we have the following properties~\cite[Rmk. 5.13]{BK2}:
\begin{equation}
	\GlK= \bigsqcup_{\tau \in \Gal(K_{e}/K)} \,\, (\GlK)^{\tau},
	\label{decomposition of GlKalg into cosets of twisted algebraic closures}
\end{equation}  
\begin{equation}
	\GlKo= \bigsqcup_{\tau \in \Gal(K_{e}/K)} \,\, (\GlKo)^{\tau}.
	\label{decomposition of Glk1 into twists}
\end{equation} 
By the equations~\eqref{GlK1algtau subset DLKtauVpsiDQl} and~\eqref{decomposition of Glk1 into twists} we get:
\begin{equation}
	(\GlKo)  \subseteq \gpDL_{K}(A)_{\Ql}.
	\label{GlK1alg subset DLKVpsiAl}
\end{equation}
In addition, $\rho_\ell (\tilde\tau) \, \GlKe = (\GlK)^{\tau}$ for all $\tau.$ Hence, we obtain the following equality 
\begin{equation}
	(\GlKo)^{\id} = \GlKeo
	\label{GlKo id = GlKeo}
\end{equation}
and the following natural isomorphism:
\begin{equation}
	\GlK/ (\GlK)^{\id} \cong \Gal(K_e/K).
	\label{GlK1alg mod GlK1algid = GLeK}
\end{equation}
Since $\gpDL_{K}^{\id}(A) = \gpDL_{K_{e}}(A)$ (from Remark~\ref{rmk: relations TLG} (4)), we obtain  the inclusion
\begin{equation}
	\GlKeo \subseteq \gpDL_{K_{e}}(A)_{\Ql},
	\label{GlLe1alg subset DLLeA}
\end{equation} 
and the following natural isomorphisms:
\begin{equation}\label{eq: iso Gl, DL, Gal}
	\GlKo/ (\GlKo)^{\id} \, \cong \, \gpDL_{K}(A)/ \gpDL_{K}^{\id}(A) \, \cong \, \Gal(K_e/K).
\end{equation}
\medskip

%--------------------Subsection 2.4------------------------------	

\subsection{The maps $d$ and $d^{\id}$}

	From the isomorphisms~\eqref{GlK1alg mod GlK1algid = GLeK} and~\eqref{eq: iso Gl, DL, Gal} we obtain the following exact sequences:
\begin{equation}
	1 \rightarrow \pi_{0} ((\GlK)^{\id}) \rightarrow \pi_{0} (\GlK) \rightarrow \Gal(K_{e}/K) \rightarrow 1
	\label{exact sequence 1}\end{equation}
\begin{equation}
	1 \rightarrow \pi_{0} ((\GlKo)^{\id}) \rightarrow \pi_{0} (\GlKo) \rightarrow \Gal(K_{e}/K) \rightarrow 1
	\label{exact sequence 2}
\end{equation} 
A natural question that arises is to know under which conditions exact sequences~\eqref{exact sequence 1} and~\eqref{exact sequence 2} split as direct products or as semi-direct products. Furthermore, it will be interesting to explore the consequences of having such exact sequences. 

Consider the natural homomorphisms:
\begin{equation}
	d : \,\pi_{0} (\GlKo)  \,\, \rightarrow  \,\, \pi_{0} (\gpDL_{K}(A)),
	\label{pi0 GlK1alg maps to pi0 DLKA}
\end{equation}
\begin{equation}
	d^{\id} :\, \pi_{0} ((\GlKo)^{\id}) \,\, \rightarrow \,\,  \pi_{0} (\gpDL_{K}^{\id}(A)).
	\label{pi0 GlK1alg maps to pi0 DLKidA} 
\end{equation}
If the homomorphism~\eqref{pi0 GlK1alg maps to pi0 DLKA} is an epimorphism then by the inclusion~\eqref{GlK1algtau subset DLKtauVpsiDQl} the homomorphism~\eqref{pi0 GlK1alg maps to pi0 DLKidA} is an epimorphism. Serre proved that there is natural isomorphism (cf.~\cite[Thm. 3.3]{BK1},~\cite[Thm. 4.8]{BK2})
\begin{equation}
	\pi_{0} (\GlKo) \cong \pi_{0} (\GlK).
\end{equation}
Hence the homomorphism~\eqref{pi0 GlK1alg maps to pi0 DLKA} gives a natural homomorphism
\begin{equation}
	\pi_{0} (\GlK)  \,\, \rightarrow  \,\, \pi_{0} (\gpDL_{K}(A)).
\end{equation}
Under the inclusion~\eqref{GlK1algtau subset DLKtauVpsiDQl} we have the following commutative diagram:
\begin{figure}[H]
\[
\begin{tikzcd}
1 \arrow{r}{} & \,\, \pi_{0}((\GlKo)^{\id}) \arrow{d}[swap]{d^{\id}} \arrow{r}{j_\ell}  & 
\pi_{0} (\GlKo)  \arrow{d}[swap]{d} \arrow{r}{\pi_\ell} & \Gal(K_{e}/K) 
\arrow{d}{=}[swap]{} \arrow{r}{} & 1 \\ 
1 \arrow{r}{} & \pi_{0} (\gpDL_{K}^{\id} (A))  \arrow{r}{j} & 
\pi_{0} (\gpDL_{K}(A)) \arrow{r}{\pi} & \Gal(K_{e}/K) \arrow{r}{} & 1\\
\end{tikzcd}
\]
\\[-0.3cm]
\caption{}
\label{diagram compatibility of GlK1alg with DLKA} 
\end{figure}

%--------------------Subection 2.5-----------------------------------

\subsection{Mumford--Tate conjecture and abelian varieties fully of Lefschetz type}

	Consider a complex abelian variety $A$, then we can attach to it a reductive algebraic group defined over $\Q$ - the \textit{Mumford--Tate group} associated to $A$. To define it, we need to introduce the following morphism:

\begin{equation}
	\begin{aligned}
		h: \, \mathrm{Res}_{\C/\R} \, \G_{m,\C} & \to \GL(V)_{\R}\\
			 z &\mapsto h(z),
	\end{aligned}
\end{equation}
where $V=H_1(A,\Q)$ is the first homology group. We can determine the Hodge structure decomposition of $V_\C$, which in turn, is equivalent to the data of the linear map $h(z)$. Indeed, we know that 
\begin{equation}
	V_\C=H_1(A,\C)=V^{-1,0}\oplus V^{0,-1},
\end{equation}
and 
\begin{equation}
	\begin{aligned}
		h(z): \, V^{-1,0}\oplus V^{0,-1} & \to V^{-1,0}\oplus V^{0,-1},\\
		 	v &\mapsto z v \quad {\rm{for}} \,\, v \in V^{-1,0}\\
		 	v &\mapsto \bar{z} v \quad {\rm{for}} \,\, v \in V^{0,-1}.
	\end{aligned}
\end{equation}
The Mumford--Tate group $\MT(A)$ is then defined as the smallest algebraic group, defined over $\Q$, such that $h$ factors through $\MT(A)_{\R}$. Notice that the Mumford--Tate group is connected. There is smaller algebraic group called the \textit{Hodge group} associated to $A$ and defined by
\begin{equation}
	\Hg(A)=(\MT(A)\cap \SL(V))^{\circ}. 
\end{equation}
\noindent
Both groups are related via the following equality:
\begin{equation}
	\MT(A)=\G_m\cdot \Hg(A).
\end{equation}
Moreover, the Hodge group $\Hg(A)$ is contained in $C_D(\Sp_{(V,\psi)}).$ Because $\Hg(A)$ is connected we have the following inclusion:
\begin{equation}
\Hg(A) \subset \mathcal{L}(A).
\end{equation}

\begin{definition}\label{def: FLT}
The abelian variety $A$ is \emph{fully of Lefschetz type} if and only if
	\begin{enumerate}
		\item $\Hg(A) \, = \, \mathcal{L}(A),$
		\item the Mumford--Tate Conjecture hold for $A$.
	\end{enumerate}
\end{definition}

\begin{remark}
By work of Deligne, Piatetski-\v{S}apiro and Borovo\v\i\, we know that 
\begin{equation}
	\GlKo^{\circ}\subseteq \Hg(A)_{\Ql} \subseteq \mathcal{L}(A)_{\Ql} = \gpDL_K(A)^{\circ}_{\Ql}.
\end{equation}
\end{remark}

\noindent
Therefore $A$ is fully of Lefschetz type if and only if the following isomorphisms:
\begin{equation}
	\GlKo^{\circ} \simeq \Hg(A)_{\Ql} \simeq \mathcal{L}(A)_{\Ql},
\label{MT-conjecture via Hg(A)}
\end{equation}
hold for every prime number $\ell.$ The first isomorphisms in \eqref{MT-conjecture via Hg(A)} is the Mumford--Tate conjecture. We refer the readers to the following survey paper~\cite{Moonen} for further details about the Mumford--Tate conjecture.

%--------------------Section 3-----------------------------------

\section{Proof of the main result and further applications}\label{sec: section 3}

%--------------------Subsection 3.1------------------------------	

\subsection{Preliminaries}\label{sec: subsection 3.1}

	The goal of this section is to determine $\pi_0(\gpDL_{K}(A))$ in terms of the endomorphism field extension $K_e$. In order to do so, we need to consider two possibilities: either $\gpDL_K^{id}(A)$ is connected or not.

\begin{remark} 
For a fixed embedding $\iota:\Ql \hookrightarrow \C$ recall that we have from the work of Serre and~\cite[Lem. 2.8]{FKRS12}, the following isomorphisms:
\begin{equation}\label{pi}
	\pi_0(\GlK) \simeq \pi_0 (\GlKo) \simeq_{\iota} \pi_0(\GlKo_\C) \simeq \pi_0(\ST_K(A)).
\end{equation}
where $\ST_K(A)$ is a maximal compact subgroup of $\GlKo(\C) = \GlKo_\C (\C).$
\label{remark-isomorphism without assumption AST}
\end{remark}
\medskip

In~\cite[Chap. 8]{Se12} Serre stated the algebraic Sato--Tate conjecture which was later on developed in~\cite[Conj. 2.1]{BK1} and~\cite[Conj. 5.1]{BK2} as follows.

\begin{conjecture}({\rm{Algebraic Sato--Tate conjecture}})   

\noindent
${\rm{(a)}}$ There is a natural-in-$K$ reductive algebraic group $\AST_K(A)$ defined over $\Q$ such 
that $\AST_K(A)\subseteq \Sp_{(V,\psi)}$ and for each prime number $\ell$ there is a natural-in-$K$ monomorphism
of group schemes:
\begin{equation}
\gpast_{\ell, K} \,\, : \,\,   \GlKo  \,\, {\stackrel{}{\hookrightarrow}} \,\, \AST_K(A)_{\Ql}.
\label{Algebraic Sato-Tate monomorphism}
\end{equation} 
\noindent
${\rm{(b)}}$ The map~\eqref{Algebraic Sato-Tate monomorphism} is an isomorphism: 
\begin{equation}
\gpast_{\ell, K} \,\, : \,\,  \GlKo  \,\, {\stackrel{\simeq}{\longrightarrow}} \,\, \AST_K(A)_{\Ql}.
\label{Algebraic Sato-Tate equality generalized}
\end{equation} 
\label{conj: alg. Sato-Tate}
\end{conjecture}

\begin{definition}\label{definition of Sato-Tate group}
The Sato--Tate group $\ST_K(A)$ is a maximal compact subgroup of $\AST_K(A) (\C).$
\end{definition}
		
\begin{remark} Assume Mumford--Tate conjecture for the abelian variety $A.$ By~\cite{CC2019} Algebraic Sato--Tate conjecture holds for $A.$  Hence under the assumption that $d:\pi_0(\GlKo) \to \pi_0(\gpDL_{K}(A))$ is an isomorphism and~\cite[Prop. 3.5]{BK1}, we have the following isomorphisms :
\begin{equation}\label{eq:isomorphisms}
	\pi_0 (\AST_K(A)) \simeq \pi_0(\ST_K(A)) \simeq  \pi_0 (\GlKo) \simeq  \pi_0(\gpDL_{K}(A)).
\end{equation}
\label{remark-isomorphism with assumption AST}
\end{remark}

\begin{remark}
Assume that $A$ is fully of Lefschetz type and $C_D(\Sp_{(V,\psi)})$ is connected i.e.
$\mathcal{L}(A) := C_D(\Sp_{(V,\psi)}).$ Then by~\cite{CC2019} and~\cite[Cor. 9.9]{BK1} the equality~\eqref{eq:isomorphisms} also holds.
\label{remark-isomorphism with assumption fully of Lefschetz type}
\end{remark}
\medskip

We will determine the component group of the twisted Lefschetz group $\gpDL_K(A)$ in terms of the component group of $\gpDL_{K}^{\id} (A)$ and the Galois group $\Gal(K_e/K)$. Firstly, we need to establish a relation between the component groups $\pi_0(\gpDL_{K}(A))$ and $\pi_0(\gpDL_{K}^{\id}(A))$ knowing that $\gpDL_{K}(A)^{\circ}=\gpDL_{K}^{\id}(A)^{\circ}$.
Secondly, from~\cite[Thm. 4.6]{BK2} and equality~\eqref{GlKo id = GlKeo} we have the following isomorphisms:
\begin{equation}\label{eq:iso Galois gps}
	\GlKo/\GlKo^{\id} \simeq \GlKo/\GlKeo \simeq \GlK/\GlKe \simeq \Gal(K_e/K),
\end{equation}
where $\GlKo^{\id}=\{g\in \GlKo\; \, g\beta g^{-1}=\beta \quad \forall \beta \in D\}$.

%--------------------Subsection 3.2------------------------------	

\subsection{The main result}\label{sec: subsection 3.2}

	The goal of this section is to prove Theorem~\ref{th: theorem intro}. Let us consider an abelian variety defined over a number field $K$. Firstly, observe that if $s_\ell$ is a splitting homomorphism of $\pi_\ell$ in Diagram~\ref{diagram compatibility of GlK1alg with DLKA with the sections}  then $d \circ s_\ell$ is a splitting homomorphism of $\pi$.
If the homomorphism $d$ in Diagrams~\ref{diagram compatibility of GlK1alg with DLKA with the sections} is an isomorphism and $\sigma_\ell$ is a splitting homomorphism of $j_\ell$ then, $\sigma := d^{\id} \circ \sigma_\ell \circ d^{-1}$ is a splitting homomorphism of $j$. Until the end of Section~\ref{sec: section 3} we will work with splitting 
$\sigma_\ell$ rather than $s_\ell$ because in principal our work concerns non-abelian groups (cf. Appendix~\ref{sec: section 4}).
  
\begin{figure}[H]
\[
\begin{tikzcd}
1 \arrow{r}{} & \,\, \pi_{0}((\GlKo)^{\id}) \arrow{d}[swap]{d^{\id}} \arrow{r}{j_\ell}  & 
\pi_{0} (\GlKo)  \arrow{d}[swap]{d} \arrow{r}{\pi_\ell} \arrow[l,bend right=33,"\sigma_\ell"{name=U, above}]& \Gal(K_{e}/K) 
\arrow{d}{=}[swap]{} \arrow{r}{} \arrow[l,bend right=33,"s_\ell"{name=U, above}]& 1 \\ 
1 \arrow{r}{} & \pi_{0} (\gpDL_{K}^{\id} (A))  \arrow{r}{j} & 
\pi_{0} (\gpDL_{K}(A)) \arrow{r}{\pi} & \Gal(K_{e}/K) \arrow{r}{} & 1\\
\end{tikzcd}
\]
\\[-0.4cm]
\caption{}
\label{diagram compatibility of GlK1alg with DLKA with the sections} 
\end{figure}

\begin{theorem}\label{th: 3.4}
Consider an abelian variety $A$ defined over a number field $K$. Assume that in Diagram~\ref{diagram compatibility of GlK1alg with DLKA with the sections} homomorphism $d$ is an epimorphism and $\sigma_{\ell}$ splits $j_{\ell}.$ Then, there exists $\sigma$ that splits $j$ and there is the following direct product decomposition:
\begin{equation}
	\pi_{0} (\gpDL_{K}(A))= \pi_{0} (\gpDL_{K}^{\id} (A))  \times \Gal(K_{e}/K).
	\label{DLK splits into DLKid times GKeK}
\end{equation}
\label{Conditions for DLK splitting into DLKid times GKeK}
\end{theorem}

\begin{proof}
	Assume that $d$ is an epimorphism and that $\sigma_{\ell}$ splits $j_{\ell}.$  Hence, for every $\beta \in \pi_{0} (\gpDL_{K}(A))$ there is $\alpha \in \pi_{0} (\GlKo)$ such that $d (\alpha) = \beta.$ Define the following map: 

\begin{equation}
	\begin{aligned}
		\sigma :\, \pi_{0} (\gpDL_{K} (A)) &\to \pi_{0} (\gpDL_{K}^{\id} (A))\\
			\beta & \mapsto \sigma (\beta) := d^{\id} \circ \sigma_\ell (\alpha).
	\end{aligned}
\end{equation}

\begin{figure}[H]
\[
\begin{tikzcd}
1 \arrow{r}{} & \,\, \pi_{0}((\GlKo)^{\id}) \arrow{d}[swap]{d^{\id}} \arrow{r}{j_\ell}  & 
\pi_{0} (\GlKo)  \arrow{d}[swap]{d} \arrow{r}{\pi_\ell} 
\arrow[l,bend right=33,"\sigma_\ell"{name=U, above}]& \Gal(K_{e}/K) 
\arrow{d}{=}[swap]{} \arrow{r}{} & 1 \\ 
1 \arrow{r}{} & \pi_{0} (\gpDL_{K}^{\id} (A))  \arrow{r}{j} & 
\pi_{0} (\gpDL_{K}(A)) \arrow{r}{\pi} \arrow[l,bend left=33,"\sigma"{name=U, below}] & 
\Gal(K_{e}/K) \arrow{r}{} & 1
\end{tikzcd}
\]
\\[0cm]
\caption{}
\label{diagram compatibility of GlK1alg with DLKA with all the sections} 
\end{figure}

We will prove that $\sigma (\beta)$ does not depend on the choice of $\alpha,$ $\sigma$ is a well defined 
homomorphism which splits $j.$ Indeed, if $d (\alpha) = 0,$ then 
$\pi (d (\alpha)) = 0.$ From Diagram~\ref{diagram compatibility of GlK1alg with DLKA with all the sections} we read that
$\pi_{\ell} (\alpha) = 0.$ Hence $\alpha = j_{\ell} (\alpha_{0})$ for some $\alpha_{0} \in 
\pi_{0}((G_{\ell, K, 1}^{\alg})^{\id}).$ Hence 
$$j \circ d^{\id} \circ \sigma_{\ell} (\alpha) = j \circ d^{\id} \circ \sigma_{\ell} \circ j_{\ell} (\alpha_{0})
= j \circ d^{\id} ((\alpha_{0})) = d \circ j_{\ell} (\alpha_{0}) = d(\alpha) = 0.$$
Because $j$ is a monomorphism then $\sigma (0) = d^{\id} \circ \sigma_{\ell} (\alpha) = 0.$
Hence $\sigma (\beta)$ does not depend on $\alpha.$ In the same way as above we prove that $\sigma$ is 
a homomorphism.
\medskip

Let us verify that $\sigma$ splits $j.$ Indeed from the property of Diagram~\ref{diagram compatibility of GlK1alg with DLKA with all the sections} and the assumption that $d$ is an epimorphism, it is clear that $d^{\id}$ is an epimorphism. Since $\sigma_\ell$ is an epimorphism, $\sigma$ is an epimorphism as well. Moreover, from the definition of $\sigma$ we obtain $\sigma \circ d =
d^{\id} \circ \sigma_\ell.$ Now observe that 
$$\sigma \circ j \circ d^{\id} \circ \sigma_\ell = \sigma \circ d \circ j_\ell \circ \sigma_\ell
= d^{\id} \circ \sigma_\ell \circ j_\ell \circ \sigma_\ell =  d^{\id} \circ \sigma_\ell.$$
Therefore, for all $\alpha \in \pi_{0} (\GlKo)$ we obtain that
$$
\sigma \circ j \circ d^{\id} \circ \sigma_\ell(\alpha)=d^{\id} \circ \sigma_\ell(\alpha),
$$
and hence, since $d^{\id} \circ \sigma_\ell$ is an epimorphism we obtain that $\sigma \circ j=\id$. This implies that $\sigma$ splits $j$. Thus, the homomorphism 
$$
(\sigma , \pi): \,  \pi_{0} (\gpDL_{K}(A)) \rightarrow \pi_{0} (\gpDL_{K}^{\id} (A))  \times \Gal(K_{e}/K)
$$
gives the splitting~\eqref{DLK splits into DLKid times GKeK}.
\end{proof}

We can deduce Theorem~\ref{th: theorem intro} from Theorem~\ref{th: 3.4}. In addition, notice that under the assumption that $A$ is fully of Lefschetz type we have the following isomorphisms 
\begin{equation}
	\pi_0(\gpDL_{K}(A)) \simeq \pi_0(\GlKo_{\Ql}) \simeq\pi_0(\AST_K(A)_{\Ql}) \simeq \pi_0(\ST_K(A)).
\end{equation}
Notice also that $\pi_0(\gpDL_{K}(A)_{\Ql}) = \pi_0(\gpDL_{K}(A)).$
Moreover, Theorem~\ref{th: 3.4} allows us to obtain the following corollary which enable us to determine concretely the component group of $\gpDL_{K}(A)_{\Ql}$ in terms the component group of $\gpDL_K^{id}(A)$ and $\Gal(K_e/K)$.

\begin{corollary}
Under the hypotheses of Theorem~\ref{th: 3.4} and the assumption that $A$ is fully of Lefschetz type, we have the following isomorphisms:
\begin{equation}
	\pi_0(\gpDL_{K}(A)) \simeq \pi_0(\ST_K(A))  \simeq \pi_0 \gpDL_K^{id}(A) \times \Gal(K_e/K).
\label{computing pi0 DLK A - general case}
\end{equation}
In particular if $\gpDL_K^{id}(A)$ is connected then we have 
\begin{equation}
	\pi_0(\gpDL_{K}(A)) \simeq \pi_0(\ST_K(A)) \simeq \Gal(K_e/K),
	\label{computing pi0 DLK A - special case}
\end{equation}
\end{corollary}

We can reformulate the above-mentioned corollary in terms of the endomorphism algebra of a simple abelian variety: 

\begin{corollary}
Under the hypotheses of Theorem~\ref{th: 3.4} an the assumption that $A$ is a simple abelian variety, fully of Lefschetz type, we have the following isomorphisms:
\begin{enumerate}
\item If $A$ is of type I, II or IV - in the sense of Albert's classification - we have:
\begin{equation}
\pi_0(\gpDL_{K}(A)) \simeq \pi_0(\ST_K(A)) \simeq \Gal(K_e/K).	
\end{equation}
\item If $A$ is of type III - in the sense of Albert's classification - we have:
\begin{equation}		
\pi_0(\gpDL_{K}(A)) \simeq \pi_0(\ST_K(A))  \simeq \pi_0 \gpDL_K^{id}(A) \times \Gal(K_e/K).
\end{equation}
\end{enumerate}
\end{corollary}
\begin{proof}
	Let us remark that the group $\gpDL_{K}^{id}(A)$ is connected for all simple abelian varieties of type I, II and IV and it is not connected for simple abelian varieties of type III, in the sense of Albert's classification. For instance, see the table at the end of~\cite[\S 2]{MilneLC} for further details. Indeed, $\gpDL_{K}^{id}(A)=C_D(\Sp)$ which corresponds to the reductive group $S(A)$ according to the notations of Milne. In that setting, when the group $\gpDL_{K}^{id}(A)$ is not connected we know that $\gpDL_{K}^{id}(A)$ is several copies of the orthogonal group. More precisely, recall that in type III, the dimension of $A$ is given by $\dim(A)=g=2eh$, where $h$ is the relative dimension of $A$. Therefore, the group $\gpDL_{K}^{id}(A)$ is given by $e$ copies of the orthogonal group $\rm{O}_{2h}$.
\end{proof}

%--------------------Subsection 3.3------------------------------	

\subsection{A conjecture of Serre and further results}\label{sec: subsection 3.3}

	The purpose of this section is to illustrate the fact that we can obtain similar diagrams as Diagram~\ref{diagram compatibility of GlK1alg with DLKA} but with the algebraic Sato--Tate group and the Sato--Tate group of the abelian variety $A$ introduced in Section~\ref{sec: subsection 3.1}. Similarly to~\eqref{eq:eq 2.8} we define the following objects	for every $\tau\in\Gal(K_e/K):$
\begin{equation}
	\begin{aligned}
		\AST_{K}^{\tau}(A) &:= \AST_{K}(A) \cap \gpDL_{K}^{\tau}(A),\\
		\ST_{K}^{\tau}(A) &:= \ST_{K}(A) \cap \AST_{K}^{\tau}(A) (\C).
	\end{aligned}
\end{equation}
\medskip

\noindent
Comparing this definition with the one of $\gpDL_{K}(A)$ we obtain:
\begin{equation}
	\AST_{K}(A)  \,\,\, = \bigsqcup_{\tau \in \Gal(K_{e}/K)} \,\, \AST_{K}^{\tau}(A),
	\label{decomposition into twisted Lefschetz for fixed elements}
\end{equation} 
\begin{equation}
	\ST_{K}(A) \,\,\, = \bigsqcup_{\tau \in \Gal(K_{e}/K)} \,\, \ST_{K}^{\tau}(A).
	\label{decomposition of AST into twisted forms for fixed elements}
\end{equation} 
Recall from Conjecture~\ref{conj: alg. Sato-Tate} that, for abelian varieties $A$ defined over a number field $K$, it is known that for every primer $\ell$, there exists a natural-in-$K$ reductive algebraic group $\AST_{K}(A)\subset \Sp_{(V,\psi)}$ defined over $\Q$ and a natural-in-$K$ monomorphism 
$\gpast_{\ell, K} \,\, : \,\,   \GlKo  \,\, {\stackrel{}{\hookrightarrow}} \,\, \AST_K(A)_{\Ql}$
(see~\eqref{Algebraic Sato-Tate monomorphism} in Conjecture~\ref{conj: alg. Sato-Tate}). Hence we have the following commutative diagram.
\begin{figure}[H]
\[
\begin{tikzcd}
1 \arrow{r}{} & \,\, \pi_{0}((\GlKo)^{\id}) \arrow{d}[swap]{d^{\id}} \arrow{r}{j_\ell}  & 
\pi_{0} (\GlKo)  \arrow{d}[swap]{d} \arrow{r}{\pi_\ell} & \Gal(K_{e}/K) 
\arrow{d}{=}[swap]{} \arrow{r}{} & 1 \\ 
1 \arrow{r}{} & \pi_{0}(\AST_{K}^{\id}(A))  \arrow{r}{j} & 
\pi_{0} (\AST_{K}(A)) \arrow{r}{\pi} & \Gal(K_{e}/K) \arrow{r}{} & 1\\
\end{tikzcd}
\]
\\[-0.5cm]
\caption{}
\label{diagram compatibility of GlK1alg with ASTKA} 
\end{figure}
\noindent 
Since $\AST_{K} (M)^{\circ} (\C)$ is a connected complex Lie group and any maximal compact subgroup of a connected complex Lie group is a connected real Lie group we have a natural isomorphism:
\begin{equation}
	\pi_{0} (\AST_{K}(A)) \cong \pi_{0} (\ST_{K}(A)).
	\label{po(AST) = po(ST)}
\end{equation} 
This isomorphism allow us to obtain the following commutative diagram:
\begin{figure}[H]
\[
\begin{tikzcd}
1 \arrow{r}{} & \,\, \pi_{0}((\GlKo)^{\id}) \arrow{d}[swap]{d^{\id}} \arrow{r}{j_\ell}  & 
\pi_{0} (\GlK)  \arrow{d} {}[swap]{d} \arrow{r}{\pi_\ell} & \Gal(K_{e}/K) 
\arrow{d}{=}[swap]{} \arrow{r}{} & 1 \\ 
1 \arrow{r}{} & \pi_{0}(\ST_{K}^{\id}(A))  \arrow{r}{j} & 
\pi_{0} (\ST_{K}(A)) \arrow{r}{\pi} & \Gal(K_{e}/K) \arrow{r}{} & 1\\
\end{tikzcd}
\]
\\[-0.5cm]
\caption{}
\label{diagram compatibility of GlK1alg with STKA} 
\end{figure}

Moreover, computations and the diagram in~\cite[p. 27]{BK2} and~\cite[Thm. 11.8, (11.12)]{BK2} allow us to obtain the following commutative diagram:
\begin{figure}[H]
\[
\begin{tikzcd}
\pi_{0}(\GlKo) \arrow{d}[swap]{\cong} \arrow{r}{d}  & 
\pi_{0} (\AST_{K}(A))  \arrow{d}[swap]{\cong}\\ 
\pi_{0}(\GlK)  \arrow{r}{} & 
\pi_{0} ( \MMT_{K}(A) )\\
\end{tikzcd}
\]
\\[-0.5cm]
\caption{}
\label{diagram comparing maps GlK1alg to ASTKA and GlKalg to MMTKA} 
\end{figure}
\noindent
Serre conjectured~\cite[9.2? p. 386]{Se94} that the bottom horizontal arrow in the Diagram~\ref{diagram comparing maps GlK1alg to ASTKA and GlKalg to MMTKA} is an epimorphism and this conjecture can be rephrased  as follows: 

\begin{conjecture} (J.-P. Serre)
	The following natural map is an epimorphism:
	\begin{equation}
		d: \pi_{0} (\GlKo)  \rightarrow \pi_{0} (\AST_{K}(A))
		\label{The map from GlK1alg to ASTKA}
	\end{equation} 
	\label{Serre Conjecture on image of Galois meeting all components}
\end{conjecture}

\noindent
One can strengthen Serre's Conjecture~\ref{Serre Conjecture on image of Galois meeting all components} to
expect more:

\begin{conjecture} (Weak Algebraic Sato--Tate Conjecture)
	The map~\eqref{The map from GlK1alg to ASTKA} is an isomorphism.
	\label{Weak Algebraic Sato-Tate Conjecture}
\end{conjecture}

\begin{remark}
	It is obvious that the Algebraic Sato--Tate Conjecture~\ref{conj: alg. Sato-Tate} implies Weak Algebraic Sato--Tate Conjecture~\ref{Weak Algebraic Sato-Tate Conjecture}.
\end{remark}

Similarly as for Theorem~\ref{th: 3.4} we can determine $\pi_{0} (\AST_{K}(A))$ (resp. $\pi_{0} (\ST_{K}(A))$) in terms of $\Gal(K_e/K)$ and $\pi_{0} (\AST_{K}^{\id}(A))$ (resp. $\pi_{0} (\ST_{K}^{\id}(A))$).

\begin{theorem}
	Assume that Serre's Conjecture~\ref{Serre Conjecture on image of Galois meeting all components} holds and 
there exists $\sigma_{\ell}$ that splits $j_{\ell}.$ Then, there is $\sigma$ that splits $j$ in each of Diagrams~\ref{diagram compatibility of GlK1alg with ASTKA} and~\ref{diagram compatibility of GlK1alg with STKA} and we have the following direct product decompositions:
	\begin{equation}
		\pi_{0} (\AST_{K}(A))= \pi_{0} (\AST_{K}^{\id}(A))  \times \Gal(K_{e}/K),
		\label{ASTK splits into ASTKid times GKeK}
	\end{equation}
	\begin{equation}
		\pi_{0} (\ST_{K}(A))= \pi_{0} (\ST_{K}^{\id}(A))  \times \Gal(K_{e}/K).
		\label{STK splits into STKid times GKeK}
	\end{equation}
\label{Conditions for ASTK splitting into ASTKid times GKeK and STK splitting into STKid times GKeK}
\end{theorem}

\begin{proof}
The proof is the same as the proof of Theorem~\ref{Conditions for DLK splitting into DLKid times GKeK}. 
\end{proof}

\begin{corollary}
	Assume that the Weak Algebraic Sato--Tate Conjecture~\ref{Weak Algebraic Sato-Tate Conjecture} holds and 
that there exists $\sigma_{\ell}$ splits $j_{\ell}.$  Then, we have the following direct products decompositions:
$$
\pi_{0} (\AST_{K}(A))= \pi_{0} (\AST_{K}^{\id}(A))  \times \Gal(K_{e}/K),
$$
$$
\pi_{0} (\ST_{K}(A))= \pi_{0} (\ST_{K}^{\id}(A))  \times \Gal(K_{e}/K).
$$
\end{corollary}

\begin{proof} It follows from Theorem~\ref{Conditions for ASTK splitting into ASTKid times GKeK and STK splitting into STKid times GKeK}.
\end{proof}

%--------------------Appendix. Remarks on splitting exact sequences--------------	
\appendix
\section{Remarks on splitting exact sequences}\label{sec: section 4}

The goal of this appendix is to establish some remarks regarding the splitting behavior of an exact sequence in the case of both, arbitrary groups and abelian groups. Recall that in section~\ref{sec: subsection 3.2} we studied this behavior for non-abelian groups.

\subsection{Arbitrary groups}
Consider the following exact sequence of groups with a splitting homomorphism $\sigma$ i.e.
$\sigma \circ j = \text{Id}_{K}.$ 
\begin{figure}[H]
\[
\begin{tikzcd}
    1\arrow{r} & K \arrow{r}{j} & G \arrow{r}{\pi}\arrow[bend left=33]{l}{\sigma} & C \arrow{r} & 1
\end{tikzcd}
\]
\\[0.0cm]
\caption{}
\label{splitting exact sequence non-abelian case} 
\end{figure}
Splitting $\sigma$ gives natural isomorphism:
\begin{equation}
	\begin{aligned}
		(\sigma, \pi): \, G \, &\to \, K \times C \\
		g & \mapsto (\sigma, \pi) (g) = (\sigma(g), \pi (g)).
	\end{aligned}
\end{equation}

Indeed, if $(\sigma(g), \pi (g)) = (1, 1)$ then there is $k \in K$ such that $j (k) = g.$ Hence $1 = \sigma (g) = \sigma j (k) = k.$ Therefore $g = 1$ so $(\sigma, \pi)$ is a monomorphism.
\medskip

\noindent
On the other hand take any $(k, c) \in K \times C$ and take $g \in G$ such that $\pi (g) = c.$ Then one checks that 
$$(\sigma, \pi) (\, j (k) \, g \, j (\sigma (g^{-1})) \, ) \, = \, (k, c).$$
Hence $(\sigma, \pi)$ is an epimorphism. 
\medskip

Observe that splitting $s$ of the form
\begin{figure}[H]
\[
\begin{tikzcd}
    1\arrow{r} & K \arrow{r}{j} & G \arrow{r}{\pi} & C \arrow{r} \arrow[bend left=33]{l}{s} & 1
\end{tikzcd}
\]
\\[0.0cm]
\caption{}
\label{the splitting s-the exact sequence} 
\end{figure}

\noindent
provides the action of the group $C$ on the group $K:$  

\begin{equation}
	\begin{aligned}
		C \times K &\to K\\
		(c, \, k) &\mapsto c(k):= j^{-1} (s(c) \, j(k) \, s(c^{-1}))
	\end{aligned}
\end{equation} 

\noindent
Identifying $j(K)$ with $K$ we can write this action in a simpler way: 
$$c(k) \, := \, s(c) \, k \, s(c^{-1}).$$

\noindent
This action leads directly to the following isomorphism: 

\begin{equation}
	\begin{aligned}
		(j, s) : \,  K \rtimes C \, &\to G\\
		(k,c) & \mapsto (j, s) (k, c) \, := \, j (k) \, s (c).
	\end{aligned}
\end{equation}

\noindent
Indeed $j (k) \, s (c) = 1$ implies $1 = \pi (j (k) \, s (c)) = \pi s (c)) = c.$ Hence $j (k) = 1$ so $k = 1.$
Therefore $(j, s)$ is a monomorphism. 
\medskip

\noindent
Take $g \in G$ and set $c := \pi (g).$ Observe that $\pi (g \, s (c^{-1}) = 1.$ Hence there is $k \in K$ such that
$j (k) = g \, s (c^{-1})$ so $j (k) \, s (c) = g.$ Therefore $(j, s)$ is an epimorphism. 

\begin{example} Let $F$ be a field. Consider thee following well known exact sequence with splitting 
$s$ defined as follows: 

\begin{figure}[H]
\[
\begin{tikzcd}
    1\arrow{r} & \SL_n (F)  \arrow{r}{} & \GL_{n} (F) \arrow{r}{\det} & F^{\times} \arrow{r} \arrow[bend left=33]{l}{s} & 1
\end{tikzcd}
\]
\\[0.0cm]
\caption{}
\label{matrix example of the splitting s} 
\end{figure}
\label{Example-matrix example of the splitting s}
\end{example}
$$s (a) \, := \, \begin{bmatrix}
a & 0 & 0 & \dots & 0\\
0 & 1 & 0 & \dots & 0\\
\vdots & \vdots & \vdots & \vdots & \vdots\\
0 & 0 & 0 & \dots & 1
\end{bmatrix}
$$
It is known that the exact sequence~\eqref{matrix example of the splitting s} does not split as a direct product 
for $n > 1.$
Hence there is no splitting homomorphism $\sigma$ of the embedding $\SL (F) \rightarrow \GL (F).$

\subsection{Abelian groups}

If the group $G$ in the following exact sequence is abelian, hence $K$ and $C$ are also abelian,
then the existence of the splitting homomorphism $\sigma$ is equivalent to the existence of splitting homomorphism 
$s.$ 
\begin{figure}[H]
\[
\begin{tikzcd}
    1\arrow{r} & K \arrow{r}{j} & G \arrow{r}{\pi}\arrow[bend left=33]{l}{\sigma} & C \arrow{r}
		\arrow[bend left=33]{l}{s}& 1
\end{tikzcd}
\]
\\[0.0cm]
\caption{}
\label{splitting exact sequence abelian case} 
\end{figure}

\noindent
Indeed if $s$ is a splitting homomorphism then we define: 
$$\sigma (g) := j^{-1} (g \, s (\pi (g^{-1}))).$$
Direct checking shows that $\sigma$ is a well defined splitting homomorphism of $j.$
\medskip

\noindent
On the other hand if $\sigma$ is a splitting homomorphism of $j$ then define $s$ as follows.
For $c \in C$ take $g$ such that $\pi (g) = c.$ Then put:
$$s (c) := g \, j (\sigma (g^{-1})).$$
This definition does not depend on the choice of $g.$ Indeed if $\pi (g^{\prime}) = c$ then
$$g \, j (\sigma (g^{-1})) = g^{\prime} \, j (\sigma ((g^{\prime})^{-1}))$$
is equivalent to
$$j (\sigma (g^{-1} g^{\prime})) = g^{-1} g^{\prime}$$
which holds true because $\pi (g^{-1} g^{\prime}) = 1.$ Hence $s$ is well defined. 
Direct checking shows that $s$ is a splitting homomorphism of $\pi.$

%-----------------BIBLIOGRAPHY----------

\end{document}